\documentclass{article}

\usepackage[a4paper, margin=1in]{geometry}
\usepackage[english]{babel}
\usepackage{csquotes}
\usepackage[backend=biber, sorting=none]{biblatex}
\usepackage{amsmath, amssymb, amsfonts, amsthm, cancel}
\usepackage{algorithm, algpseudocode}
\usepackage{graphicx}
\usepackage{booktabs, longtable, adjustbox}
\usepackage{caption}
\usepackage{siunitx}
\usepackage{enumitem}
\usepackage{authblk}
\usepackage[colorlinks=true, urlcolor=blue, citecolor=blue, linkcolor=blue, bookmarks=false]{hyperref}

\theoremstyle{plain}
\newtheorem{theorem}{Theorem}
\newtheorem{lemma}{Lemma}

\theoremstyle{definition}
\newtheorem{definition}{Definition}
\newtheorem{assumption}{Assumption}

\theoremstyle{remark}
\newtheorem{remark}{Remark}

\newcommand{\R}{\mathbb{R}}
\newcommand{\norm}[1]{\|#1\|}
\DeclareMathOperator{\tr}{tr}
\newcommand{\PG}{\Pi_{\Gamma}}
\newcommand{\PI}[1]{\Pi_{#1}}
\newcommand{\keywords}[1]{\noindent\textit{Keywords:} #1}

\newcommand{\algoname}[1]{{\small\texttt{#1}}}
\newcommand{\GMOPCGM}{\algoname{GMOPCGM}}
\newcommand{\GCGPM}{\algoname{GCGPM}}
\newcommand{\MOPCGM}{\algoname{MOPCGM}}
\newcommand{\CGPM}{\algoname{CGPM}}
\newcommand{\STTDFPM}{\algoname{STTDFPM}}

\newcommand{\thetaG}{\theta_k^{\scriptscriptstyle G}}     
\newcommand{\thetaC}{\hat{\theta}_k}                       
\newcommand{\tstar}{t_k^*}                                 

\title{Spectral conjugate gradient projection methods for large-scale\\
monotone equations without Lipschitz continuity}
\author[1,5]{Kabenge Hamiss}
\author[1,4]{Mohammed Alshahrani}
\author[2,3]{Mujahid N.\ Syed}
\affil[1]{\small Department of Mathematics, King Fahd University of Petroleum and Minerals, Dhahran 31261, Saudi Arabia}
\affil[2]{\small Department of Industrial \& Systems Engineering, King Fahd University of Petroleum and Minerals, Dhahran 31261, Saudi Arabia}
\affil[3]{\small Interdisciplinary Research Center for Intelligent Secure Systems, King Fahd University of Petroleum and Minerals, Dhahran 31261, Saudi Arabia}
\affil[4]{\small Interdisciplinary Research Center for Smart Mobility and Logistics, King Fahd University of Petroleum and Minerals, Dhahran 31261, Saudi Arabia}
\affil[5]{\small Department of Mathematics and Statistics, Islamic University in Uganda, Mbale 2555, Uganda}

\begin{document}
\maketitle

\begin{abstract}
We introduce two derivative-free projection methods for large-scale systems of nonlinear monotone equations subject to convex constraints. Both methods incorporate an adaptive spectral parameter into established conjugate gradient frameworks: the first generalizes the modified optimal Perry method via an eigenvalue-optimized scaling matrix, and the second generalizes the Hager--Zhang-type conjugate gradient projection method via a spectral Dai--Liao parameter. The resulting search directions satisfy a sufficient descent condition independent of the line search. For the first method, we establish global convergence under monotonicity alone, without requiring Lipschitz continuity of the mapping. For the second, global convergence holds under the standard monotonicity and Lipschitz continuity assumptions. Numerical experiments on 18 test problems across dimensions up to 120{,}000, together with applications to $\ell_1$-regularized signal recovery and regularized logistic regression, confirm the practical effectiveness of the proposed approach.
\end{abstract}

\keywords{Nonlinear monotone equations, convex constraints, spectral conjugate gradient method, derivative-free method, projection method, global convergence, compressed sensing, logistic regression.}

\medskip
\noindent\textit{MSC 2020:} 65H10, 65K05, 90C30, 90C56

\section{Introduction}\label{sec:intro}

We consider the problem of finding a vector $x \in \Gamma$ satisfying
\begin{equation}\label{eq:problem}
  G(x) = 0,
\end{equation}
where $G \colon \R^n \to \R^n$ is a continuous and monotone mapping and $\Gamma \subset \R^n$ is a nonempty, closed, and convex set. Systems of this form arise in diverse areas of science and engineering. Variational inequality problems and complementarity problems can often be reformulated as~\eqref{eq:problem} \cite{he2002new, malitsky2014extragradient, zhao2001monotonicity}. Further applications include chemical equilibrium computations~\cite{meintjes1987methodology, zeleznik1968calculation}, compressed sensing and signal restoration~\cite{figueiredo2007gradient, abubakar2020note}, regularized logistic regression and machine learning~\cite{ibrahim2025relaxed, chorowski2014learning}, and problems in fluid and plasma physics~\cite{chen2019gramian, gao2019mathematical}.

When~$G$ is differentiable, Newton and quasi-Newton methods~\cite{nocedal1999numerical, iusem1997newton} offer rapid local convergence. However, they require computing or approximating the Jacobian of~$G$ at each iteration, which entails $O(n^2)$ storage and substantial computational cost per step. For large-scale problems, or when $G$ is not differentiable, these methods become impractical.

Conjugate gradient (CG) methods provide a natural alternative. They require only function evaluations and $O(n)$ storage, making them well suited for large-scale problems. Classical CG parameters---Fletcher--Reeves~\cite{fletcher1964function}, Polak--Ribi\`{e}re--Polyak~\cite{polyak1969conjugate}, Hestenes--Stiefel~\cite{hestenes1952methods}, and Dai--Yuan~\cite{dai1999nonlinear}---have given rise to a rich family of CG-based derivative-free methods for solving~\eqref{eq:problem}; see~\cite{cheng2009prp, li2011class, liu2015projection, xiao2013conjugate} and the references therein. The spectral gradient method of Zhang and Zhou~\cite{zhang2006spectral}, which incorporates Barzilai--Borwein step sizes~\cite{barzilai1988two}, has proven especially effective for monotone equations. The combination of CG update formulas with spectral parameters has led to spectral conjugate gradient methods~\cite{birgin2000nonmonotone, abubakar2022hybrid}, which often exhibit improved numerical performance.

A key ingredient for extending these methods to the constrained setting~$\Gamma \neq \R^n$ is the projection technique of Solodov and Svaiter~\cite{solodov1999globally}. The approach proceeds in three stages at each iteration: (i)~compute a descent direction~$p_k$ and find a step size~$\alpha_k$ via a line search, yielding the trial point $z_k = x_k + \alpha_k p_k$; (ii)~construct a hyperplane that separates the current iterate from the solution set, using the monotonicity of~$G$; (iii)~project onto this hyperplane to obtain the next iterate~$x_{k+1}$. This framework has become the standard approach for designing derivative-free projection methods~\cite{yu2009spectral, wang2016self, gao2018efficient, zheng2020modified}.

Three-term CG methods, which generate directions of the form $p_k = -\theta_0 G_k + \theta_1 p_{k-1} + \theta_2 w_{k-1}$ for suitable vectors~$w_{k-1}$ and scalars~$\theta_i$, have received considerable attention due to their favorable descent properties and robustness; see~\cite{zhang2007some, andrei2013three, narushima2011three} for the unconstrained setting and~\cite{zheng2020conjugate, ibrahim2023two, liu2022three, liu2025threeterm} for the constrained monotone setting. Recently, several authors have combined three-term CG structures with spectral parameters and the Solodov--Svaiter projection to obtain methods with improved theoretical and computational properties~\cite{waziri2022double, abubakar2022hybrid, waziri2020descent}.

Among these methods, three serve as the foundation and benchmarks for our work. Sabi'u et al.~\cite{sabi2023modified} proposed the \MOPCGM, which derives the CG update from a modified optimal Perry matrix via eigenvalue analysis and establishes global convergence under monotonicity and Lipschitz continuity. Zheng et al.~\cite{zheng2020conjugate} proposed the \CGPM, which adapts the Hager--Zhang~\cite{hager2005new} conjugate gradient parameter to the constrained monotone setting and proves convergence under similar assumptions. Ibrahim et al.~\cite{ibrahim2023two} proposed the \STTDFPM, a spectral three-term method that achieves global convergence without requiring Lipschitz continuity, by using a pseudomonotonicity assumption.

In this paper, we introduce an adaptive spectral scaling parameter~$\lambda_k$ into both the \MOPCGM{} and \CGPM{} frameworks. The parameter~$\lambda_k$ is bounded within an interval $[\alpha_{\min}, \alpha_{\max}]$ and adapts to the local curvature of the problem, drawing on ideas from the Barzilai--Borwein spectral gradient family. The resulting methods, which we call \GMOPCGM{} and \GCGPM, retain the sufficient descent property of their predecessors while offering additional flexibility through the spectral parameter. The contributions of this paper can be summarized as follows.
\begin{enumerate}[label=(\roman*), nosep]
  \item We derive two new search directions that generalize the \MOPCGM{} and \CGPM{} directions via a spectral parameter~$\lambda_k$, and we show that both satisfy a sufficient descent condition independent of the line search.
  \item We prove global convergence of \GMOPCGM{} under Assumptions~\ref{A1} and~\ref{A2} alone (nonemptiness of the solution set and monotonicity), without requiring Lipschitz continuity of~$G$.
  \item We prove global convergence of \GCGPM{} under the standard Assumptions~\ref{A1}--\ref{A3} (including Lipschitz continuity).
  \item We present numerical experiments on 18 test problems, an application to compressed sensing, and an application to regularized logistic regression on real-world datasets that demonstrate the practical advantages of the spectral generalization.
\end{enumerate}

The remainder of this paper is organized as follows. Section~\ref{sec:prelim} collects the notation, definitions, and standing assumptions. Section~\ref{sec:methods} derives the two proposed methods and establishes their descent properties. Section~\ref{sec:convergence} presents the global convergence analysis. Section~\ref{sec:experiments} reports the numerical experiments. Sections~\ref{sec:application} and~\ref{sec:logreg} present applications to compressed sensing and regularized logistic regression, respectively. Section~\ref{sec:conclusion} concludes the paper.

\section{Preliminaries}\label{sec:prelim}

Throughout this paper, $\norm{\cdot}$ denotes the Euclidean norm, $\langle \cdot, \cdot \rangle$ denotes the standard inner product in~$\R^n$, and we write $G_k = G(x_k)$ for brevity. The projection operator onto the closed convex set~$\Gamma$ is defined by
\[
  \PG(x) = \arg\min_{y \in \Gamma} \norm{y - x}, \quad x \in \R^n.
\]

\begin{definition}\label{def:mono}
A mapping $G \colon \R^n \to \R^n$ is said to be \emph{monotone} if
\[
  \langle G(x) - G(y),\, x - y \rangle \geq 0, \quad \forall\, x, y \in \R^n.
\]
\end{definition}

The projection operator satisfies the following well-known properties~\cite{solodov1999globally}:
\begin{enumerate}[label=(\alph*), nosep]
  \item $\langle x - \PG(x),\, y - \PG(x) \rangle \leq 0$ for all $x \in \R^n$ and $y \in \Gamma$;
  \item $\norm{\PG(x) - \PG(y)} \leq \norm{x - y}$ for all $x, y \in \R^n$ (non-expansiveness);
  \item $\norm{\PG(x) - y}^2 \leq \norm{x - y}^2 - \norm{x - \PG(x)}^2$ for all $x \in \R^n$ and $y \in \Gamma$.
\end{enumerate}

We make the following standing assumptions.

\begin{assumption}\label{A1}
The solution set $S^* = \{x \in \Gamma : G(x) = 0\}$ is nonempty.
\end{assumption}

\begin{assumption}\label{A2}
The mapping $G$ is monotone on $\R^n$.
\end{assumption}

\begin{assumption}\label{A3}
The mapping $G$ is Lipschitz continuous on $\R^n$, i.e., there exists a constant $L > 0$ such that $\norm{G(x) - G(y)} \leq L\norm{x - y}$ for all $x, y \in \R^n$.
\end{assumption}

Since the algorithms evaluate~$G$ at each iteration, continuity of~$G$ is implicitly required; this is guaranteed by Assumption~\ref{A3} when it is in force, and holds for all problems considered in this paper. Assumption~\ref{A3} is used only in the convergence analysis of \GCGPM{} and in the alternative convergence proof of \GMOPCGM. The primary convergence result for \GMOPCGM{} requires only Assumptions~\ref{A1} and~\ref{A2}.

\section{The proposed methods}\label{sec:methods}

We derive the two proposed search directions and verify their descent properties. Throughout, we use the notation
\begin{equation}\label{eq:notation}
  s_{k-1} = z_{k-1} - x_{k-1}, \qquad
  y_{k-1} = G_k - G_{k-1}, \qquad
  v_{k-1} = y_{k-1} + \tau s_{k-1},
\end{equation}
for a fixed parameter $\tau > 0$.

We begin with the generalized modified optimal Perry conjugate gradient method. Following the quasi-Newton approach of Perry~\cite{perry1978modified} and Sabi'u et al.~\cite{sabi2023modified}, consider the search direction $p_k = -\tilde{Q}_k G_k$, where
\begin{equation}\label{eq:Qmatrix}
  \tilde{Q}_k = \lambda I
    - \frac{\lambda}{2}\,\frac{y_{k-1}s_{k-1}^T}{y_{k-1}^Ts_{k-1}}
    - \frac{\lambda}{2}\,\frac{s_{k-1}y_{k-1}^T}{y_{k-1}^Ts_{k-1}}
    + t_k\,\frac{s_{k-1}s_{k-1}^T}{y_{k-1}^Ts_{k-1}},
\end{equation}
with $\lambda > 0$ and $t_k > 0$. Since $G$ is monotone, $s_{k-1}^T y_{k-1} > 0$ whenever $x_k \neq x^*$, so both $s_{k-1}$ and $y_{k-1}$ are nonzero. For any vector~$a$ in the subspace spanned by $\{s_{k-1}, y_{k-1}\}$,
\[
  a^T \tilde{Q}_k\, a = t_k\,\frac{(a^T s_{k-1})^2}{y_{k-1}^T s_{k-1}} > 0,
\]
confirming positive definiteness on this subspace. The matrix $\tilde{Q}_k$ is a rank-2 perturbation of $\lambda I$, so $\lambda$ is an eigenvalue of multiplicity $n-2$. Let $\eta_k^+$ and $\eta_k^-$ denote the remaining two eigenvalues. We determine the optimal~$t_k$ through the following eigenvalue analysis. For readability, set
\begin{equation}\label{eq:ab}
  a \;=\; \frac{\norm{s_{k-1}}^2}{s_{k-1}^T y_{k-1}}, \qquad
  b \;=\; \frac{\norm{s_{k-1}}\,\norm{y_{k-1}}}{s_{k-1}^T y_{k-1}}.
\end{equation}

\begin{lemma}\label{lem:trace}
Let $\tilde{Q}_k$ be defined by~\eqref{eq:Qmatrix}. Then
\begin{equation}\label{eq:trace}
  \tr(\tilde{Q}_k) \;=\; \lambda(n-1) + a\,t_k.
\end{equation}
\end{lemma}
\begin{proof}
By linearity of the trace, $\tr(y_{k-1}s_{k-1}^T) = s_{k-1}^T y_{k-1}$, and $\tr(s_{k-1}s_{k-1}^T) = \norm{s_{k-1}}^2$:
\[
  \tr(\tilde{Q}_k) = n\lambda - \tfrac{\lambda}{2} - \tfrac{\lambda}{2} + a\,t_k
  = \lambda(n-1) + a\,t_k. \qedhere
\]
\end{proof}

\begin{lemma}\label{lem:tracesq}
Let $\tilde{Q}_k$ be defined by~\eqref{eq:Qmatrix}. Then
\begin{equation}\label{eq:tracesq}
  \tr\!\left(\tilde{Q}_k^T \tilde{Q}_k\right)
  \;=\; \lambda^2\!\left(n - \tfrac{3}{2}\right)
    + \tfrac{\lambda^2}{2}\,b^2
    + a^2 t_k^2.
\end{equation}
\end{lemma}
\begin{proof}
Expanding $\tilde{Q}_k^T \tilde{Q}_k$ term by term, applying the trace identities $\tr(y_{k-1}s_{k-1}^T) = s_{k-1}^T y_{k-1}$, $\tr(s_{k-1}s_{k-1}^T) = \norm{s_{k-1}}^2$, and $\tr(y_{k-1}y_{k-1}^T) = \norm{y_{k-1}}^2$, one obtains
\[
  \tr\!\left(\tilde{Q}_k^T \tilde{Q}_k\right)
  = n\lambda^2 - \tfrac{3}{2}\lambda^2
    + \cancel{t_k\lambda\,a}
    + \tfrac{\lambda^2}{2}\,b^2
    - \cancel{t_k\lambda\,a}
    + a^2 t_k^2,
\]
which simplifies to~\eqref{eq:tracesq}.
\end{proof}

\begin{lemma}\label{lem:eigenproduct}
The eigenvalues $\eta_k^+$ and $\eta_k^-$ satisfy
\begin{align}
  \eta_k^+ + \eta_k^- &= \lambda + a\, t_k, \label{eq:eigsum}\\
  \eta_k^+\, \eta_k^- &= \tfrac{\lambda^2}{4}(1 - b^2) + \lambda\, a\, t_k. \label{eq:eigprod}
\end{align}
\end{lemma}
\begin{proof}
From Lemma~\ref{lem:trace}, $\lambda(n{-}2) + \eta_k^+ + \eta_k^- = \lambda(n{-}1) + a\,t_k$, giving~\eqref{eq:eigsum}. From Lemma~\ref{lem:tracesq}, $(\eta_k^+)^2 + (\eta_k^-)^2 = \frac{\lambda^2}{2}(1 + b^2) + a^2 t_k^2$. Applying $\eta_k^+\eta_k^- = \frac{1}{2}[(\eta_k^+ + \eta_k^-)^2 - ((\eta_k^+)^2 + (\eta_k^-)^2)]$ yields~\eqref{eq:eigprod}.
\end{proof}

\begin{remark}\label{rem:specialcase1}
Setting $\lambda = 1$ recovers $\eta_k^+\eta_k^- = \frac{1}{4}(1 - b^2) + a\,t_k$, the expression in~\cite{sabi2023modified}.
\end{remark}

\begin{lemma}\label{lem:optimalt}
The value of~$t_k$ minimizing the condition number of~$\tilde{Q}_k$ is
\begin{equation}\label{eq:tstar}
  \tstar = \frac{\lambda}{a} = \lambda\,\frac{s_{k-1}^T y_{k-1}}{\norm{s_{k-1}}^2}.
\end{equation}
\end{lemma}
\begin{proof}
From~\eqref{eq:eigsum}--\eqref{eq:eigprod}, $\eta_k^\pm$ satisfy $\eta^2 - (\lambda + a\,t_k)\eta + \frac{\lambda^2}{4}(1-b^2) + \lambda a\,t_k = 0$, so
\[
  (\eta_k^+ - \eta_k^-)^2 = (a\,t_k - \lambda)^2 + \lambda^2(b^2 - 1).
\]
Since $b \geq 1$ by Cauchy--Schwarz, this is minimized when $t_k = \lambda/a$.
\end{proof}

\begin{remark}
For $\lambda = 1$, we recover $\tstar = s_{k-1}^T y_{k-1}/\norm{s_{k-1}}^2$, the optimal Perry parameter in~\cite{sabi2023modified}.
\end{remark}

To allow $\lambda$ to adapt to the problem and iterate, we define the spectral parameter
\begin{equation}\label{eq:lambdak}
  \lambda_k = \PI{[\alpha_{\min},\, \alpha_{\max}]}\!\left(\max\!\left\{
    \frac{\norm{s_{k-1}}^2}{s_{k-1}^T v_{k-1}},\;
    \frac{s_{k-1}^T v_{k-1}}{\norm{v_{k-1}}^2}
  \right\}\right),
\end{equation}
where $\PI{[a,b]}(x) = \max\{a, \min\{x, b\}\}$ and $0 < \alpha_{\min} \leq \alpha_{\max}$. These are Barzilai--Borwein-type spectral ratios~\cite{barzilai1988two}; the projection keeps $\lambda_k$ bounded.

The \GMOPCGM{} search direction is then
\begin{equation}\label{eq:GMOPCGM}
  p_k = \begin{cases}
    -G_k, & k = 0, \\[6pt]
    -M_k\, G_k + \thetaG\, p_{k-1}, & k \geq 1,
  \end{cases}
\end{equation}
where
\begin{align}
  M_k &= \lambda_k + \thetaG\,\frac{G_k^T p_{k-1}}{\norm{G_k}^2}, \label{eq:Mk}\\[4pt]
  \thetaG &= \frac{(v_{k-1} - \tstar\, s_{k-1})^T G_k}{p_{k-1}^T v_{k-1}}, \label{eq:thetaG}\\[4pt]
  \tstar &= \lambda_k\,\frac{s_{k-1}^T v_{k-1}}{\norm{s_{k-1}}^2}. \label{eq:tstarv}
\end{align}

\begin{lemma}\label{lem:descentGMOP}
Let $\{p_k\}$ and $\{G_k\}$ be generated by Algorithm~\ref{alg:GMOPCGM}. Then
\begin{equation}\label{eq:descentGMOP}
  G_k^T p_k \leq -\alpha_{\min}\norm{G_k}^2.
\end{equation}
\end{lemma}
\begin{proof}
Multiplying~\eqref{eq:GMOPCGM} by $G_k^T$:
\[
  G_k^T p_k
  = -\lambda_k\norm{G_k}^2
    - \thetaG\,\frac{G_k^T p_{k-1}}{\norm{G_k}^2}\,\norm{G_k}^2
    + \thetaG\, G_k^T p_{k-1}
  = -\lambda_k\norm{G_k}^2.
\]
Since $\lambda_k \geq \alpha_{\min}$ by~\eqref{eq:lambdak}, the result follows.
\end{proof}

The descent condition~\eqref{eq:descentGMOP} is independent of~$\tstar$.

\begin{algorithm}[ht]
\caption{Generalized Modified Optimal Perry Conjugate Gradient Method (\GMOPCGM)}
\label{alg:GMOPCGM}
\begin{algorithmic}[1]
  \State Choose $x_0 \in \R^n$, $\varepsilon > 0$, $\rho \in (0,1)$, $\beta > 0$, $\zeta > 0$, $\tau > 0$,
  \Statex \hspace{1.5em} $0 < \alpha_{\min} \leq \alpha_{\max}$, $0 < \zeta_1 \leq \zeta_2$, $\gamma \in (0,2)$. Set $k \gets 0$.
  \While{$\norm{G_k} > \varepsilon$}
    \State Compute $p_k$ by~\eqref{eq:GMOPCGM}.
    \State Find $\alpha_k = \max\{\rho^i \beta : i = 0,1,2,\ldots\}$ satisfying
    \begin{equation}\label{eq:linesearch1}
      G(x_k + \alpha_k p_k)^T p_k \leq -\zeta\,\alpha_k\norm{p_k}^2\,\PI{[\zeta_1,\zeta_2]}(\norm{G(x_k + \alpha_k p_k)}).
    \end{equation}
    \State Set $z_k \gets x_k + \alpha_k p_k$.
    \If{$z_k \in \Gamma$ and $\norm{G(z_k)} \leq \varepsilon$}
      \State \Return $x^* \gets z_k$.
    \EndIf
    \State Set $\mu_k \gets \frac{G(z_k)^T(x_k - z_k)}{\norm{G(z_k)}^2}$ and $x_{k+1} \gets \PG(x_k - \gamma\mu_k G(z_k))$.
    \State Compute $s_k$, $v_k$, $\thetaG$, $\tstar$ by~\eqref{eq:thetaG}--\eqref{eq:tstarv}.
    \State Update $\lambda_{k+1}$ by~\eqref{eq:lambdak}.
    \State $k \gets k+1$.
  \EndWhile
  \State \Return $x^* \gets x_k$.
\end{algorithmic}
\end{algorithm}

We now present the second method, which generalizes the \CGPM{} of Zheng et al.~\cite{zheng2020conjugate}. The Hager--Zhang parameter~\cite{hager2005new} is a special case of the Dai--Liao parameter~\cite{dai2012another} with $t = 2\norm{y_{k-1}}^2/(s_{k-1}^T y_{k-1})$. We generalize this to
\begin{equation}\label{eq:newtk}
  t_k = \lambda\,\frac{\norm{y_{k-1}}^2}{s_{k-1}^T y_{k-1}}, \qquad \lambda > 0.
\end{equation}
The \GCGPM{} search direction is
\begin{equation}\label{eq:GCGPM}
  p_k = \begin{cases}
    -G_k, & k = 0, \\[6pt]
    -\lambda_k\, G_k + \thetaC\, p_{k-1} + \tau\, a_k\, w_{k-1}, & k \geq 1,
  \end{cases}
\end{equation}
with the following auxiliary quantities:
\begin{align}
  w_{k-1} &= y_{k-1} + r_k\, p_{k-1}, &
  r_k &= 1 + \max\!\left\{0,\, -\frac{G_k^T p_{k-1}}{y_{k-1}^T p_{k-1}}\right\}, \label{eq:wk}\\[4pt]
  a_k &= \frac{G_k^T p_{k-1}}{w_{k-1}^T p_{k-1}}, &
  \lambda_k &= \PI{[\alpha_{\min},\alpha_{\max}]}\!\left(\max\!\left\{
    \frac{\norm{w_{k-1}}^2}{p_{k-1}^T w_{k-1}},\;
    \frac{p_{k-1}^T w_{k-1}}{\norm{p_{k-1}}^2}
  \right\}\right), \label{eq:lambdak2}\\[4pt]
  \thetaC &= \frac{G_k^T w_{k-1}}{p_{k-1}^T w_{k-1}}
    - \lambda_k\,\frac{\norm{w_{k-1}}^2}{p_{k-1}^T w_{k-1}}
    \cdot\frac{G_k^T p_{k-1}}{p_{k-1}^T w_{k-1}}. \label{eq:thetaC}
\end{align}

\begin{lemma}\label{lem:descentGCGPM}
Let $\{p_k\}$ and $\{G_k\}$ be generated by Algorithm~\ref{alg:GCGPM} with $0 \leq \tau \leq 1$ and $\alpha_{\min} > (1{+}\tau)/2$. Then
\begin{equation}\label{eq:descentGCGPM}
  G_k^T p_k \;\leq\; -\alpha_{\min}\!\left(1 - \frac{(1+\tau)^2}{4\alpha_{\min}^2}\right)\norm{G_k}^2.
\end{equation}
\end{lemma}
\begin{proof}
The case $k = 0$ is immediate. For $k \geq 1$, multiplying~\eqref{eq:GCGPM} by $G_k^T$ and expanding:
\begin{align*}
  G_k^T p_k &= -\lambda_k\norm{G_k}^2
    + (1{+}\tau)\,\frac{G_k^T w_{k-1}}{p_{k-1}^T w_{k-1}}\,G_k^T p_{k-1}
    - \lambda_k\,\frac{\norm{w_{k-1}}^2(G_k^T p_{k-1})^2}{(p_{k-1}^T w_{k-1})^2}.
\end{align*}
Applying the AM--GM inequality $2uv \leq u^2 + v^2$ with
\[
  u = \frac{(1{+}\tau)}{\sqrt{2\lambda_k}}\,\norm{G_k}\,|p_{k-1}^T w_{k-1}|, \qquad
  v = \sqrt{2\lambda_k}\,\norm{w_{k-1}}\,|G_k^T p_{k-1}|,
\]
the cross term is bounded by
\[
  \frac{(1{+}\tau)^2}{4\lambda_k}\norm{G_k}^2
  + \lambda_k\,\frac{\norm{w_{k-1}}^2(G_k^T p_{k-1})^2}{(p_{k-1}^T w_{k-1})^2}.
\]
After cancellation and $\lambda_k \geq \alpha_{\min}$, we obtain~\eqref{eq:descentGCGPM}.
\end{proof}

\begin{algorithm}[ht]
\caption{Generalized Conjugate Gradient Projection Method (\GCGPM)}
\label{alg:GCGPM}
\begin{algorithmic}[1]
  \State Choose $x_0 \in \R^n$, $\varepsilon > 0$, $\rho \in (0,1)$, $\eta > 0$, $\zeta > 0$, $\tau > 0$,
  \Statex \hspace{1.5em} $0 < \alpha_{\min} \leq \alpha_{\max}$, $0 < \zeta_1 \leq \zeta_2$, $\gamma \in (0,2)$. Set $k \gets 0$.
  \While{$\norm{G_k} > \varepsilon$}
    \State Compute $\thetaC$ by~\eqref{eq:thetaC} and $p_k$ by~\eqref{eq:GCGPM}.
    \State Find $\alpha_k = \max\{\rho^i \eta : i = 0,1,2,\ldots\}$ satisfying
    \begin{equation}\label{eq:linesearch2}
      G(x_k + \alpha_k p_k)^T p_k \leq -\zeta\,\alpha_k\norm{p_k}^2\,\PI{[\zeta_1,\zeta_2]}(\norm{G(x_k + \alpha_k p_k)}).
    \end{equation}
    \State Set $z_k \gets x_k + \alpha_k p_k$.
    \If{$z_k \in \Gamma$ and $\norm{G(z_k)} \leq \varepsilon$}
      \State \Return $x^* \gets z_k$.
    \EndIf
    \State Set $\mu_k \gets \frac{G(z_k)^T(x_k - z_k)}{\norm{G(z_k)}^2}$ and $x_{k+1} \gets \PG(x_k - \gamma\mu_k G(z_k))$.
    \State Compute $w_k$, $\thetaC$ by~\eqref{eq:wk}--\eqref{eq:thetaC}.
    \State Update $\lambda_{k+1}$ by~\eqref{eq:lambdak2}.
    \State $k \gets k+1$.
  \EndWhile
  \State \Return $x^* \gets x_k$.
\end{algorithmic}
\end{algorithm}

\section{Convergence analysis}\label{sec:convergence}

We first establish the convergence properties of \GMOPCGM{} (Algorithm~\ref{alg:GMOPCGM}).

\begin{lemma}\label{lem:stepsizeGMOP}
Suppose Assumption~\ref{A3} holds. Then the step size in Algorithm~\ref{alg:GMOPCGM} satisfies
\[
  \alpha_k \geq \min\!\left\{\beta,\; \frac{\alpha_{\min}\rho}{L + \zeta\zeta_2}\cdot\frac{\norm{G_k}^2}{\norm{p_k}^2}\right\}.
\]
\end{lemma}
\begin{proof}
If $\alpha_k = \beta$, the bound is trivial. Otherwise, $\tilde{\alpha}_k = \alpha_k/\rho$ violates~\eqref{eq:linesearch1}, so
\[
  G(x_k + \tilde{\alpha}_k p_k)^T p_k > -\zeta\tilde{\alpha}_k\norm{p_k}^2\,\PI{[\zeta_1,\zeta_2]}(\norm{G(x_k + \tilde{\alpha}_k p_k)}).
\]
Using the sufficient descent condition~\eqref{eq:descentGMOP} and the Lipschitz continuity of~$G$,
\begin{align*}
  \alpha_{\min}\norm{G_k}^2 &\leq -G_k^T p_k = (G(x_k + \tilde{\alpha}_k p_k) - G_k)^T p_k - G(x_k + \tilde{\alpha}_k p_k)^T p_k \\
  &\leq L\tilde{\alpha}_k\norm{p_k}^2 + \zeta\tilde{\alpha}_k\norm{p_k}^2\zeta_2
  = \frac{\alpha_k}{\rho}(L + \zeta\zeta_2)\norm{p_k}^2. \qedhere
\end{align*}
\end{proof}

\begin{lemma}\label{lem:trustregionGMOP}
Suppose Assumptions~\ref{A1}--\ref{A3} hold. Then the direction~$p_k$ satisfies the trust region property
\[
  \alpha_{\min}\norm{G_k} \leq \norm{p_k} \leq \kappa\norm{G_k},
\]
where $\kappa = \alpha_{\max} + 2(1+\alpha_{\max})\frac{L\gamma + \tau}{\tau}$.
\end{lemma}
\begin{proof}
The lower bound follows directly from~\eqref{eq:descentGMOP} and the Cauchy--Schwarz inequality: $\alpha_{\min}\norm{G_k}^2 \leq |G_k^T p_k| \leq \norm{G_k}\norm{p_k}$.

For the upper bound, note that $s_{k-1} = \alpha_{k-1} p_{k-1}$ and $\norm{v_{k-1}} \leq (L\gamma + \tau)\norm{s_{k-1}}$. It follows that $|t_k^*| = \lambda_k\frac{|v_{k-1}^T s_{k-1}|}{\norm{s_{k-1}}^2} \leq \alpha_{\max}(L\gamma + \tau)$, and consequently
\[
  |\thetaG| \leq (1+\alpha_{\max})\frac{(L\gamma + \tau)\norm{G_k}}{\tau\norm{p_{k-1}}}.
\]
Substituting into $\norm{p_k} \leq \alpha_{\max}\norm{G_k} + 2|\thetaG|\norm{p_{k-1}}$ yields the result.
\end{proof}

\begin{lemma}\label{lem:limGMOP}
Suppose Assumptions~\ref{A1} and~\ref{A2} hold. Then $\lim_{k \to \infty} \alpha_k\norm{p_k} = 0$.
\end{lemma}
\begin{proof}
From the line search~\eqref{eq:linesearch1}, we have $G(z_k)^T p_k \leq -\zeta\alpha_k\norm{p_k}^2\norm{G(z_k)}$, which gives
\begin{equation}\label{eq:hyperplane_ineq}
  G(z_k)^T(x_k - z_k) = -\alpha_k G(z_k)^T p_k \geq \zeta\norm{x_k - z_k}^2\norm{G(z_k)}.
\end{equation}
By monotonicity of~$G$ and Assumption~\ref{A1}, there exists $x^* \in \Gamma$ with $G(x^*) = 0$, and
\begin{equation}\label{eq:mono_ineq}
  G(z_k)^T(x_k - x^*) \geq G(z_k)^T(x_k - z_k) \geq \zeta\norm{x_k - z_k}^2\norm{G(z_k)}.
\end{equation}
The non-expansiveness of the projection then gives
\begin{align*}
  \norm{x_{k+1} - x^*}^2 &\leq \norm{x_k - x^*}^2 - 2\gamma\mu_k G(z_k)^T(x_k - x^*) + \gamma^2\mu_k^2\norm{G(z_k)}^2 \\
  &\leq \norm{x_k - x^*}^2 - \gamma(2-\gamma)\frac{(G(z_k)^T(x_k - z_k))^2}{\norm{G(z_k)}^2}.
\end{align*}
Combining with~\eqref{eq:hyperplane_ineq},
\begin{equation}\label{eq:descent_seq}
  \norm{x_{k+1} - x^*}^2 \leq \norm{x_k - x^*}^2 - \gamma(2-\gamma)\zeta^2\norm{x_k - z_k}^4.
\end{equation}
The sequence $\{\norm{x_k - x^*}\}$ is therefore non-increasing and bounded below, hence convergent. Summing~\eqref{eq:descent_seq} over $k$ yields
\[
  \gamma(2-\gamma)\zeta^2 \sum_{k=0}^{\infty}\norm{x_k - z_k}^4 \leq \norm{x_0 - x^*}^2 < \infty,
\]
so $\norm{x_k - z_k} = \alpha_k\norm{p_k} \to 0$.
\end{proof}

\begin{theorem}\label{thm:convergenceGMOP}
Suppose Assumptions~\ref{A1} and~\ref{A2} hold. Let $\{x_k\}$ be generated by Algorithm~\ref{alg:GMOPCGM}. Then
\begin{equation}\label{eq:globalconvGMOP}
  \liminf_{k \to \infty}\norm{G_k} = 0.
\end{equation}
\end{theorem}
\begin{proof}
Suppose for contradiction that there exists $\varepsilon > 0$ with $\norm{G_k} > \varepsilon$ for all~$k$. By~\eqref{eq:descentGMOP}, $\norm{p_k} \geq \alpha_{\min}\varepsilon > 0$ for all~$k$. Together with Lemma~\ref{lem:limGMOP}, this implies $\alpha_k \to 0$.

For each~$k$, let $\bar{\alpha}_k = \alpha_k/\rho$, which violates~\eqref{eq:linesearch1}:
\[
  -G(x_k + \bar{\alpha}_k p_k)^T p_k < \bar{\alpha}_k\zeta\norm{p_k}^2\zeta_2.
\]
Since $\{x_k\}$ and $\{p_k\}$ are bounded (by Lemma~\ref{lem:limGMOP} and the descent inequality), we may extract subsequences converging to accumulation points $\bar{x}$ and $\bar{p}$. Taking $k \to \infty$ in the violated line search condition gives $-G(\bar{x})^T\bar{p} \leq 0$, while the sufficient descent condition gives $-G(\bar{x})^T\bar{p} \geq \alpha_{\min}\norm{G(\bar{x})}^2 > 0$, a contradiction.
\end{proof}

Under the additional Lipschitz continuity assumption, we can give an alternative proof.

\begin{theorem}\label{thm:convergenceGMOP_Lip}
Suppose Assumptions~\ref{A1}--\ref{A3} hold. Then~\eqref{eq:globalconvGMOP} holds.
\end{theorem}
\begin{proof}
If $\liminf_{k\to\infty}\norm{p_k} = 0$, the result follows from~\eqref{eq:descentGMOP}. Otherwise, $\liminf_{k\to\infty}\norm{p_k} > 0$. For each~$k$, let $\tilde{\alpha}_k = \alpha_k/\rho$ violate~\eqref{eq:linesearch1}. Using the sufficient descent condition and Lipschitz continuity,
\[
  \alpha_{\min}\norm{G_k}^2 \leq (L + \zeta\zeta_2)\tilde{\alpha}_k\norm{p_k}^2,
\]
so $\norm{G_k}^2 \leq \frac{(L+\zeta\zeta_2)}{\rho\alpha_{\min}}\alpha_k\norm{p_k}^2 \to 0$ by Lemma~\ref{lem:limGMOP}.
\end{proof}

We now turn to the convergence of \GCGPM{} (Algorithm~\ref{alg:GCGPM}).

\begin{lemma}\label{lem:stepsizeGCGPM}
Suppose Assumption~\ref{A3} holds, $0 \leq \tau \leq 1$, and $\alpha_{\min} > \frac{1+\tau}{2}$. Then
\[
  \alpha_k \geq \min\!\left\{\eta,\; \frac{\rho[4\alpha_{\min}^2 - (1+\tau)^2]}{4\alpha_{\min}(L + \zeta\zeta_2)}\cdot\frac{\norm{G_k}^2}{\norm{p_k}^2}\right\}.
\]
\end{lemma}
\begin{proof}
The proof follows the same steps as Lemma~\ref{lem:stepsizeGMOP}, with the descent condition~\eqref{eq:descentGCGPM} replacing~\eqref{eq:descentGMOP}.
\end{proof}

\begin{lemma}\label{lem:limGCGPM}
Suppose Assumptions~\ref{A1} and~\ref{A2} hold. Then $\lim_{k\to\infty}\alpha_k\norm{p_k} = 0$.
\end{lemma}
\begin{proof}
The proof follows the same structure as Lemma~\ref{lem:limGMOP}, using the line search~\eqref{eq:linesearch2} and the monotonicity of~$G$. The key inequality~\eqref{eq:descent_seq} holds identically, giving $\sum_{k=0}^\infty \norm{x_k - z_k}^4 < \infty$.
\end{proof}

\begin{theorem}\label{thm:convergenceGCGPM}
Suppose Assumptions~\ref{A1}--\ref{A3} hold, $0 \leq \tau \leq 1$, and $\alpha_{\min} > \frac{1+\tau}{2}$. Let $\{x_k\}$ be generated by Algorithm~\ref{alg:GCGPM}. Then
\[
  \liminf_{k\to\infty}\norm{G_k} = 0.
\]
\end{theorem}
\begin{proof}
If $\liminf_{k\to\infty}\norm{p_k} = 0$, the result follows from~\eqref{eq:descentGCGPM}. Otherwise, for each~$k$, let $\tilde{\alpha}_k = \alpha_k/\rho$ violate~\eqref{eq:linesearch2}. Using~\eqref{eq:descentGCGPM} and Lipschitz continuity,
\[
  \alpha_{\min}\!\left[1 - \frac{(1+\tau)^2}{4\alpha_{\min}^2}\right]\norm{G_k}^2 \leq (L + \zeta\zeta_2)\tilde{\alpha}_k\norm{p_k}^2,
\]
so $\norm{G_k}^2 \leq \frac{4\alpha_{\min}(L+\zeta\zeta_2)}{\rho[4\alpha_{\min}^2 - (1+\tau)^2]}\alpha_k\norm{p_k}^2 \to 0$ by Lemma~\ref{lem:limGCGPM}.
\end{proof}

\section{Numerical experiments}\label{sec:experiments}

All algorithms were implemented in Julia~1.12 and executed on a desktop PC with an Intel Core~i9-9900K CPU (3.60\,GHz, 8 cores) and 32\,GB RAM running Windows~11 Pro. We compare \GMOPCGM{} and \GCGPM{} against three existing methods: \MOPCGM~\cite{sabi2023modified}, \CGPM~\cite{zheng2020conjugate}, and \STTDFPM~\cite{ibrahim2023two}. All competitor methods use their originally published line searches and parameters.

The algorithms are terminated when any of the following conditions is met:
\begin{enumerate}[label=(\roman*), nosep]
  \item $\norm{G(x_k)} < \varepsilon$ with $\varepsilon = 10^{-11}$,
  \item $\norm{p_k} < 0.1\varepsilon$,
  \item $k > 2000$,
  \item the line search fails to find a valid step size,
  \item a NaN or infinity is detected in $\norm{G(x_k)}$ or~$p_k$.
\end{enumerate}
Experiments are conducted at six dimensions, $n \in \{10^3,\, 5{\times}10^3,\, 10^4,\, 5{\times}10^4,\, 10^5,\, 1.2{\times}10^5\}$, with 10 initial points for each dimension. We record the number of iterations (IT), function evaluations (FE), and CPU time in seconds (CPU).

The parameters for \GMOPCGM{} and \GCGPM{} are reported in Table~\ref{tab:params}.

\begin{table}[htbp]
\centering
\caption{Parameter values for the proposed methods.}\label{tab:params}
\begin{tabular}{lll}
\toprule
Parameter & \GMOPCGM & \GCGPM \\
\midrule
$\tau$ & 1.0 & 0.001 \\
$\rho$ & 0.8 & 0.5 \\
$\beta$ or $\eta$ & 0.5 & 0.6 \\
$\zeta$ & 0.0001 & 0.1 \\
$\alpha_{\min}$ & 0.1 & 0.55 \\
$\alpha_{\max}$ & 2.0 & 4.9 \\
$\gamma$ & 1.1 & 1.8 \\
$\zeta_1, \zeta_2$ & 1.0, 1.0 & 1.0, 1.0 \\
\bottomrule
\end{tabular}
\end{table}

Two implementation-level enhancements were applied to the proposed methods. First, the projection relaxation parameter~$\gamma$ was adapted during the iteration: when $\norm{G_{k+1}} < \norm{G_k}$, $\gamma$ was increased by a factor of~$1.1$ (capped at~$1.8$ for \GMOPCGM{} and~$1.7$ for \GCGPM), and otherwise held near its current value, always remaining in~$(0,2)$. Since the convergence proofs require only $\gamma_k \in (0,2)$ at each iteration, this adaptive strategy preserves all theoretical guarantees while improving practical performance. Second, the spectral parameter~$\lambda_k$ was updated only when $\norm{G_{k+1}} \geq c\,\norm{G_k}$ (with $c = 0.75$ for \GMOPCGM{} and $c = 0.6$ for \GCGPM), retaining the current value during rapid-convergence phases. Since $\lambda_k$ always stays in $[\alpha_{\min},\alpha_{\max}]$, the sufficient descent conditions~\eqref{eq:descentGMOP} and~\eqref{eq:descentGCGPM} remain valid.

The 18 test problems used in the experiments are listed in Table~\ref{tab:testproblems}. All problems are solved over the constraint set $\Gamma = \R^n_+ = \{x \in \R^n : x_i \geq 0\}$, except Problem~16 which uses $\Gamma = [1,\infty)^n$ since $G$ involves $\log x_i$. Ten initial points are used: eight constant vectors $x_0 = c\,e$ with $c \in \{0.4,\, 0.5,\, 0.6,\, 0.8,\, 1.0,\, 1.1,\, 2.0,\, 5.0\}$ (where $e = (1,\ldots,1)^T$), and two structured vectors $x_0 = (1,\, 1/2,\, \ldots,\, 1/n)^T$ and $x_0 = (1/n,\, 2/n,\, \ldots,\, 1)^T$.

\begin{table}[htbp]
\centering
\caption{Test problems ($i = 1,\ldots,n$ unless stated otherwise).}\label{tab:testproblems}
\small
\begin{tabular}{@{}cll@{}}
\toprule
No.\ & Source & $G_i(x)$ \\
\midrule
1 & \cite{sabi2020two}, 4.1 & $2x_i - \sin x_i$ \\
2 & \cite{la2006spectral}, 10 & $\log(x_i+1) - x_i/n$ \\
3 & \cite{zheng2020conjugate}, 4.1 & $e^{x_i} - 1$ \\
4 & \cite{sabi2023modified}, 4.5 & $4x_i + (x_{i+1}-2x_i) - x_{i-1}^2/3$ \\
5 & \cite{zheng2020conjugate}, 4.4 & $x_i - \exp(\cos(\cdots))$ \\
6 & \cite{zheng2020conjugate}, 4.4 & $-x_{i-1} + 2x_i + \sin x_i - 1$ \\
7 & \cite{song2024efficient}, 14 & $x_i(x_{i-1}^2 + 2x_i^2 + x_{i+1}^2) - 1$ \\
8 & \cite{song2024efficient}, 2 & $(x_i - 1)^2 - 1.01$ \\
9 & \cite{song2024efficient}, 4 & $(i/n)e^{x_i} - 1$ \\
10 & \cite{ibrahim2024two}, 11 & $x_i - \sin|x_i - 1|$ \\
11 & \cite{waziri2022two}, 4.5 & $2x_i - \sin|x_i - 1|$ \\
12 & \cite{song2024efficient}, 6 & $x_i - 2\sin|x_i - 1|$ \\
13 & \cite{song2024efficient}, 11 & $(e^{x_i})^2 + 3\sin x_i\cos x_i - 1$ \\
14 & \cite{waziri2020descent}, 5 & $x_{i-1} + 2.5x_i + x_{i+1} - 1$ \\
15 & \cite{zhou2007limited}, 1 & $2x_i - \sin|x_i|$ \\
16 & \cite{la2006spectral}, 32 & Minimal function \\
17 & \cite{li2021scaled}, 4.11 & $2\!\cdot\!10^{-5}(x_i-1) + 4x_i\sum x_j^2 - x_i$ \\
18 & \cite{sabi2023modified}, 4.6 & $x_i\cos(x_i - 1/n)(\sin x_i - 1 - (1-x_i)^2 - \frac{1}{n}\sum x_j)$ \\
\bottomrule
\end{tabular}
\end{table}

The total benchmark comprises $18 \times 10 \times 5 \times 6 = 5{,}400$ runs. For comparison, we use the performance profiles of Dolan and Mor\'{e}~\cite{dolan2002benchmarking}. Figures~\ref{fig:iterations}--\ref{fig:time} display the profiles for iterations, function evaluations, and CPU time, respectively.

\begin{figure}[htbp]
\centering
\includegraphics[width=0.78\textwidth]{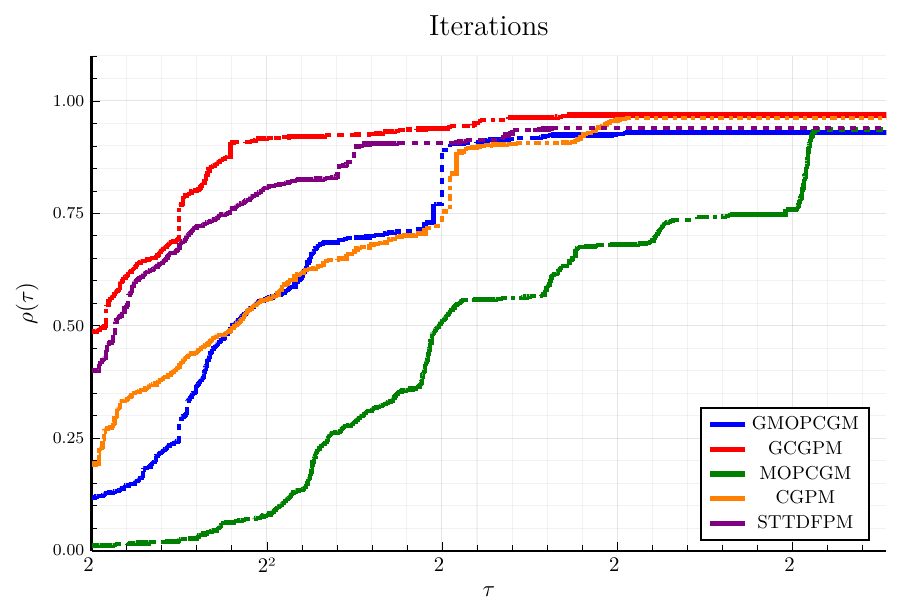}
\caption{Performance profile by number of iterations across all 5{,}400 benchmark instances. Higher curves indicate a greater probability of being within a factor~$\tau$ of the best solver.}\label{fig:iterations}
\end{figure}

\begin{figure}[htbp]
\centering
\includegraphics[width=0.78\textwidth]{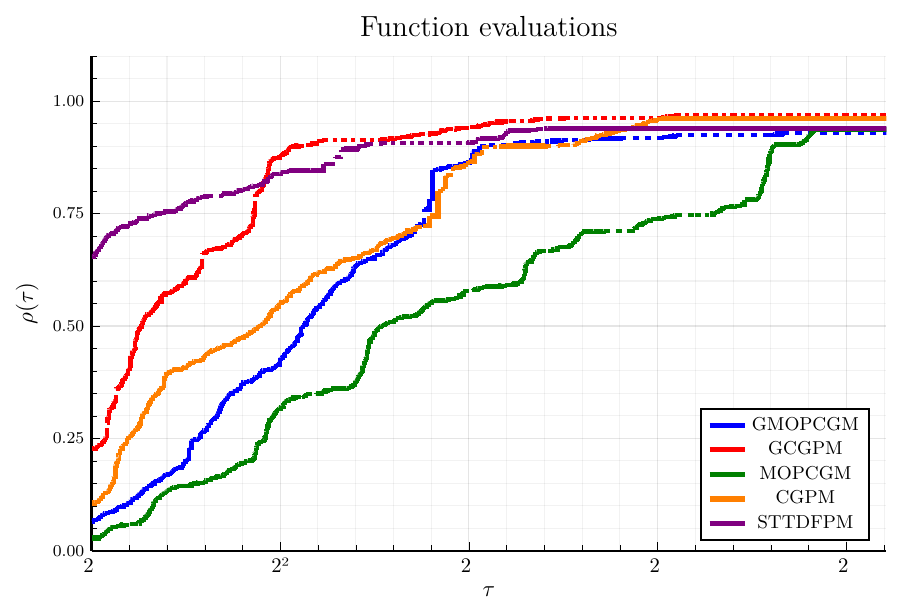}
\caption{Performance profile by number of function evaluations.}\label{fig:FE}
\end{figure}

\begin{figure}[htbp]
\centering
\includegraphics[width=0.78\textwidth]{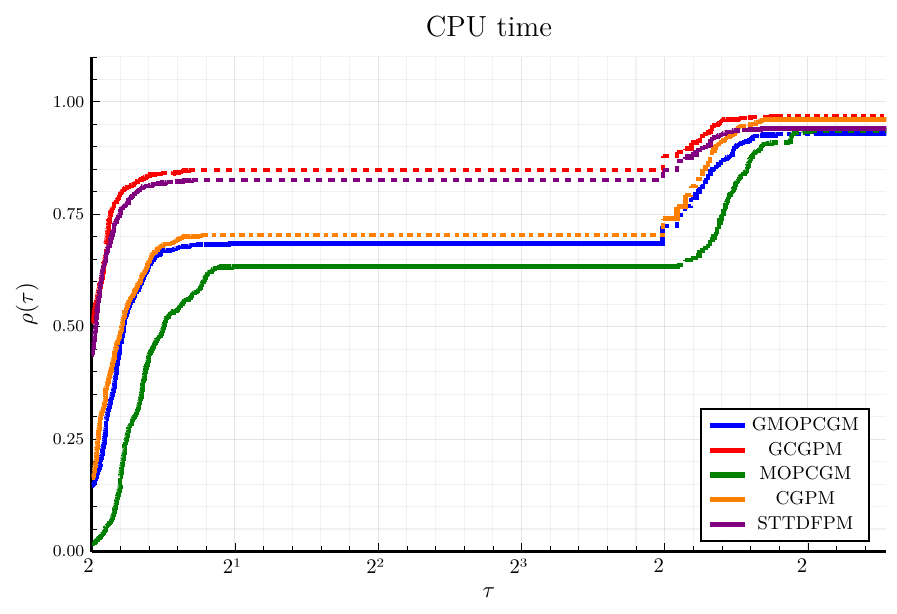}
\caption{Performance profile by CPU time (seconds).}\label{fig:time}
\end{figure}

Table~\ref{tab:aggregate} summarizes the aggregate performance. \GCGPM{} achieves the highest convergence rate (96.9\%) and the lowest median iteration count (10) and function evaluations (52) among all five methods. \GMOPCGM{} converges in 93.0\% of instances, comparable to \MOPCGM{} (93.6\%) and \STTDFPM{} (94.0\%), but with substantially fewer median iterations (28 vs.\ 160 for \MOPCGM). Both proposed methods improve markedly over their respective predecessors.

\begin{table}[htbp]
\centering
\caption{Aggregate performance across all 18 problems, 6 dimensions, and 10 initial points (5{,}400 runs per method). Median values are reported over converged instances.}\label{tab:aggregate}
\begin{tabular}{lrrrrrr}
\toprule
Method & Total & Conv & Rate (\%) & med IT & med FE & med CPU (s) \\
\midrule
\GMOPCGM & 1080 & 1004 & 93.0 & 28 & 120 & 0.029 \\
\GCGPM & 1080 & 1046 & 96.9 & 10 & 52 & 0.009 \\
\MOPCGM & 1080 & 1011 & 93.6 & 160 & 320 & 0.114 \\
\CGPM & 1080 & 1038 & 96.1 & 19 & 90 & 0.027 \\
\STTDFPM & 1080 & 1015 & 94.0 & 13 & 39 & 0.010 \\
\bottomrule
\end{tabular}
\end{table}

Table~\ref{tab:wins} presents pairwise win/tie/loss counts based on the number of iterations over all instances where both methods converged. The primary comparison for each proposed method is against its predecessor. \GMOPCGM{} wins 967 out of 982 comparisons against \MOPCGM{} (with 12 ties and only 3 losses), and \GCGPM{} wins 738 out of 1018 against \CGPM{} (with 35 ties and 245 losses). Both generalized methods dominate \MOPCGM{} overwhelmingly. Between the two proposed methods, \GCGPM{} wins 781 out of 989 comparisons against \GMOPCGM. Against \STTDFPM, which belongs to a different algorithmic family, \GCGPM{} wins 485 comparisons while losing 396, reflecting a moderate advantage in iteration count. \STTDFPM{} achieves the fewest function evaluations per iteration due to its faster-contracting line search, while \GCGPM{} requires fewer iterations overall.

\begin{table}[htbp]
\centering
\caption{Pairwise comparison by iterations (wins/ties/losses). Each entry shows the count for the row method against the column method over instances where both converged.}\label{tab:wins}
\small
\begin{tabular}{lccccc}
\toprule
 & \GMOPCGM & \GCGPM & \MOPCGM & \CGPM & \STTDFPM \\
\midrule
\GMOPCGM & --- & 163/45/781 & 967/12/3 & 536/43/419 & 260/40/678 \\
\GCGPM & 781/45/163 & --- & 967/6/14 & 738/35/245 & 485/124/396 \\
\MOPCGM & 3/12/967 & 14/6/967 & --- & 23/6/977 & 4/12/958 \\
\CGPM & 419/43/536 & 245/35/738 & 977/6/23 & --- & 315/81/607 \\
\STTDFPM & 678/40/260 & 396/124/485 & 958/12/4 & 607/81/315 & --- \\
\bottomrule
\end{tabular}
\end{table}

Figure~\ref{fig:scaling} shows how the median CPU time scales with the problem dimension. All methods exhibit near-linear growth, with \GCGPM{} and \STTDFPM{} consistently the fastest across all six dimensions. Figure~\ref{fig:convergence} displays a representative convergence trajectory on Problem~5 at $n = 50{,}000$, illustrating the residual reduction per iteration.

\begin{figure}[htbp]
\centering
\includegraphics[width=0.78\textwidth]{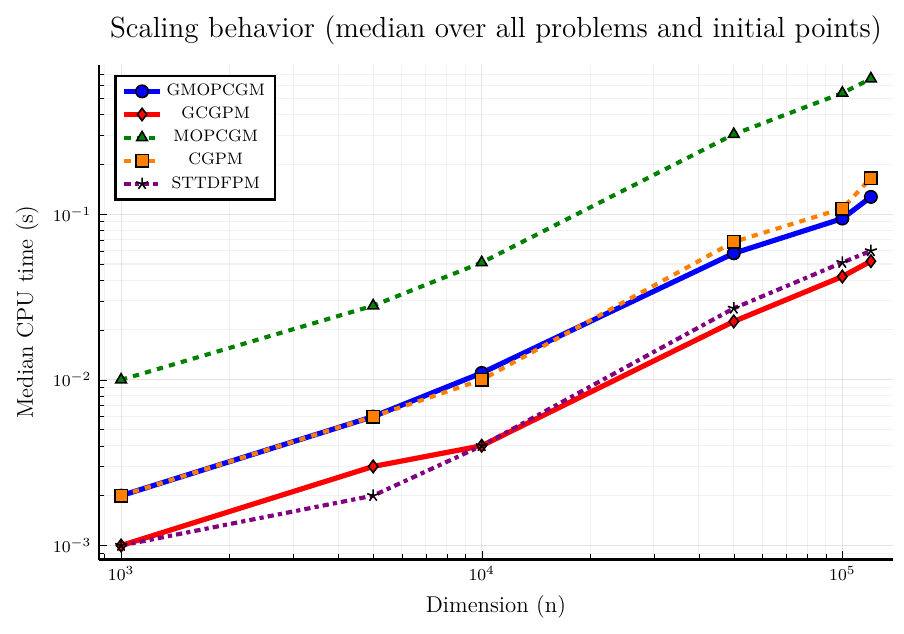}
\caption{Median CPU time versus problem dimension~$n$, aggregated over all 18 problems and 10 initial points. Both axes use a logarithmic scale.}\label{fig:scaling}
\end{figure}

\begin{figure}[htbp]
\centering
\includegraphics[width=0.78\textwidth]{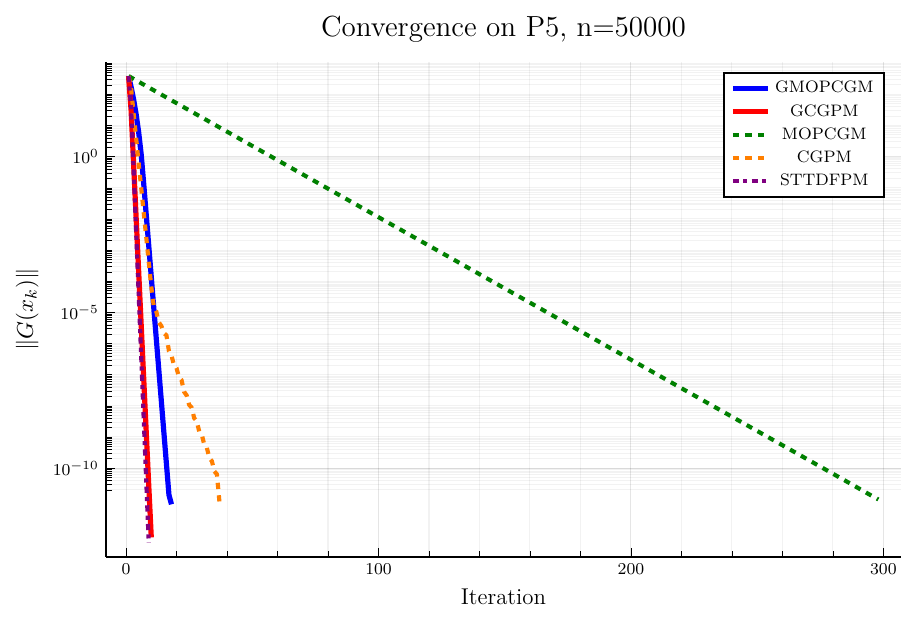}
\caption{Residual $\norm{G(x_k)}$ versus iteration on Problem~5 ($n = 50{,}000$, $x_0 = e$). \GCGPM{} converges in 9 iterations, comparable to \STTDFPM{} (8), while \MOPCGM{} requires nearly 300.}\label{fig:convergence}
\end{figure}

Detailed per-problem results for each dimension and initial point are provided in the supplementary material.

\section{Application to compressed sensing}\label{sec:application}

Signal restoration---recovering an original signal from degraded or incomplete observations---is a fundamental problem in signal processing~\cite{combettes2005signal, stark2013image}. In compressed sensing, the goal is to recover a sparse signal~$x \in \R^n$ from measurements $b = Ax + \nu$, where $A \in \R^{m \times n}$ with $m \ll n$ is a sensing matrix and $\nu$ is additive noise~\cite{figueiredo2007gradient}. This is accomplished by solving the $\ell_1$-regularized least-squares problem
\begin{equation}\label{eq:l1}
  \min_{x \in \R^n}\; \tfrac{1}{2}\norm{b - Ax}^2 + \tau\norm{x}_1,
\end{equation}
where $\tau > 0$ is a regularization parameter promoting sparsity. Following~\cite{xiao2011non, gao2018efficient, sabi2020two}, problem~\eqref{eq:l1} can be reformulated as
\begin{equation}\label{eq:complementarity}
  G(z) = \min\{z,\; Qz + c\} = 0, \quad z \geq 0,
\end{equation}
where $z = (\max\{x,0\};\, \max\{-x,0\})$, $Q = \bigl(\begin{smallmatrix} A^TA & -A^TA \\ -A^TA & A^TA \end{smallmatrix}\bigr)$, and $c = \tau\mathbf{1}_{2n} + (-A^Tb;\, A^Tb)$. The mapping~$G$ in~\eqref{eq:complementarity} is monotone and continuous~\cite{iusem1997newton}, so the problem fits the framework of~\eqref{eq:problem}.

We generate a sparse signal $x_{\mathrm{orig}} \in \R^n$ with $n = 2^{12}$ having $k$ randomly placed nonzero entries drawn from $\mathcal{N}(0, 10^{-3})$. The sensing matrix $A \in \R^{m \times n}$ is constructed by drawing i.i.d.\ $\mathcal{N}(0, 10^{-3})$ entries and orthogonalizing the rows via a QR factorization. The observation is $b = Ax_{\mathrm{orig}} + \nu$, where $\nu \sim \mathcal{N}(0, \sigma^2 I)$. The regularization parameter is set adaptively as $\tau = 0.01\norm{A^Tb}_\infty$, and each algorithm is initialized at $z_0 = (\max\{A^Tb, 0\};\, \max\{-A^Tb, 0\})$ and terminated when $\norm{G(z_k)} < 10^{-5}$. All methods use the same parameters as in the benchmark experiments.

To study how the methods perform under varying problem difficulty, we sweep over three parameters:
\begin{enumerate}[label=(\roman*), nosep]
  \item sparsity ratio $k/n \in \{0.05,\, 0.10,\, 0.20,\, 0.30\}$,
  \item measurement ratio $m/n \in \{0.25,\, 0.50,\, 0.75\}$,
  \item noise level $\sigma \in \{0,\, 10^{-3},\, 10^{-2},\, 10^{-1}\}$.
\end{enumerate}
For each of the $4 \times 3 \times 4 = 48$ configurations, we run 5 independent trials and report median values. The recovery quality is measured by the mean square error $\mathrm{MSE} = \norm{x_{\mathrm{orig}} - x_{\mathrm{rec}}}/n$.

Table~\ref{tab:cs} summarizes the median performance across all 240 instances per method. Both \GMOPCGM{} and \GCGPM{} converge on all 240 instances, as do \CGPM{} and \STTDFPM. \MOPCGM{} fails on 20 instances (all at the highest sparsity ratio $k/n = 0.30$). In terms of efficiency, \STTDFPM{} requires the fewest iterations (183) and function evaluations (410), followed closely by \GCGPM{} (197 iterations) and \GMOPCGM{} (235 iterations). \MOPCGM{} is the slowest, requiring over 1{,}300 median iterations and 16 seconds of CPU time---roughly 5--7 times more than the other methods. All methods achieve comparable recovery quality, with median MSE on the order of $10^{-6}$.

\begin{table}[htbp]
\centering
\sisetup{exponent-mode=scientific, round-mode=places, round-precision=2}
\caption{Compressed sensing: median performance over all 48 configurations and 5 trials (240 instances per method).}\label{tab:cs}
\begin{tabular}{lrrrrrr}
\toprule
Method & Total & Conv (\%) & med IT & med FE & med CPU (s) & med MSE \\
\midrule
\GMOPCGM & 240 & 100.0 & 235 & 705 & 3.02 & \num{1.23e-06} \\
\GCGPM & 240 & 100.0 & 197 & 788 & 2.92 & \num{1.22e-06} \\
\MOPCGM & 240 & 91.7 & 1362 & 2723 & 16.33 & \num{9.56e-07} \\
\CGPM & 240 & 100.0 & 712 & 1570 & 9.02 & \num{1.23e-06} \\
\STTDFPM & 240 & 100.0 & 183 & 410 & 2.37 & \num{1.23e-06} \\
\bottomrule
\end{tabular}
\end{table}

Figure~\ref{fig:cs_signals} displays a representative set of reconstructed signals from a single instance, confirming that all converging methods achieve visually indistinguishable recoveries. Figure~\ref{fig:cs_residual} shows the residual convergence for the same instance, illustrating the convergence speed differences.

\begin{figure}[htbp]
\centering
\includegraphics[width=0.78\textwidth]{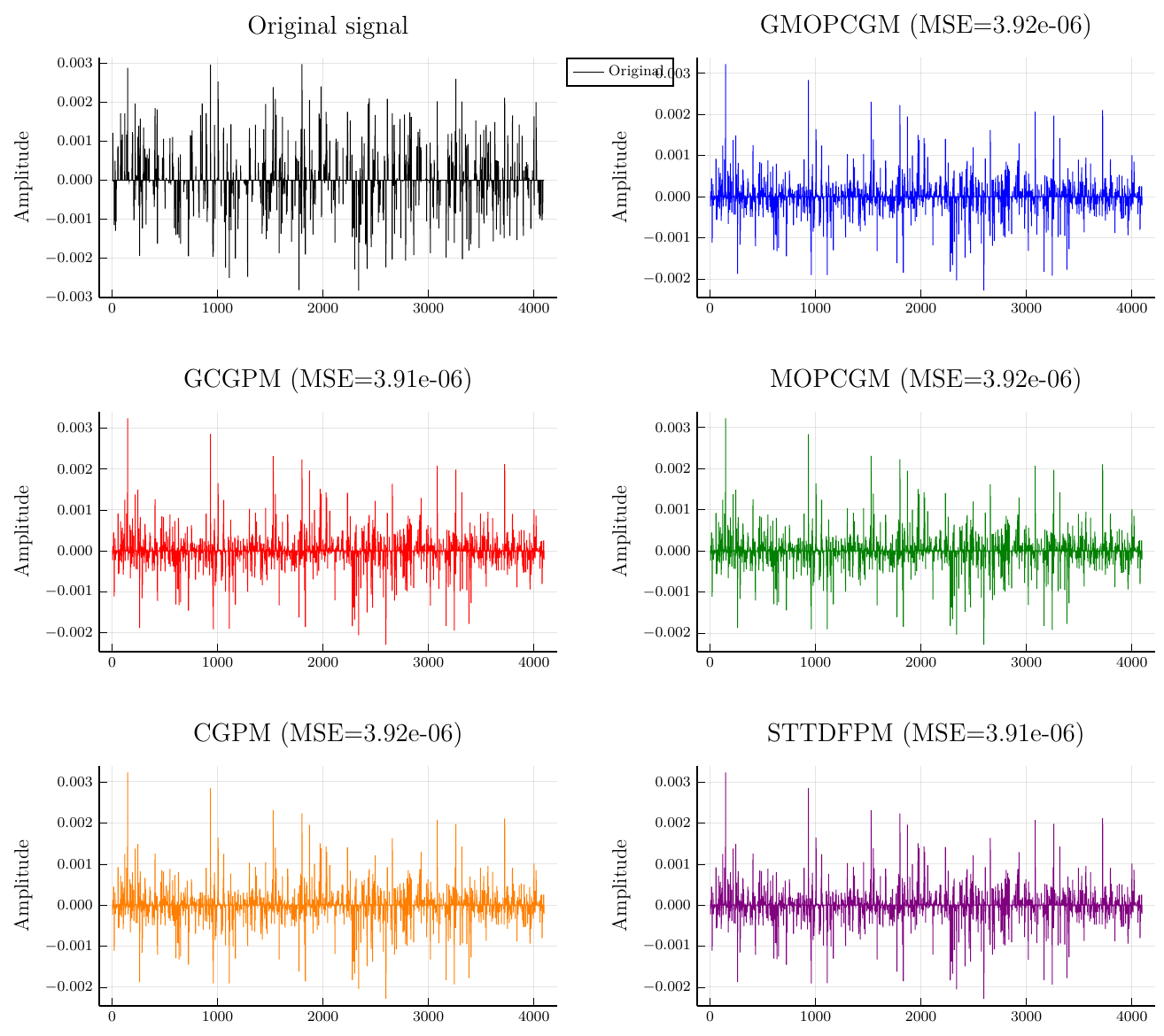}
\caption{Original and reconstructed signals for a single instance ($n = 4096$, $k = 512$, $m = 1024$, $\sigma = 10^{-4}$). All methods recover the signal with comparable fidelity.}\label{fig:cs_signals}
\end{figure}

\begin{figure}[htbp]
\centering
\includegraphics[width=0.78\textwidth]{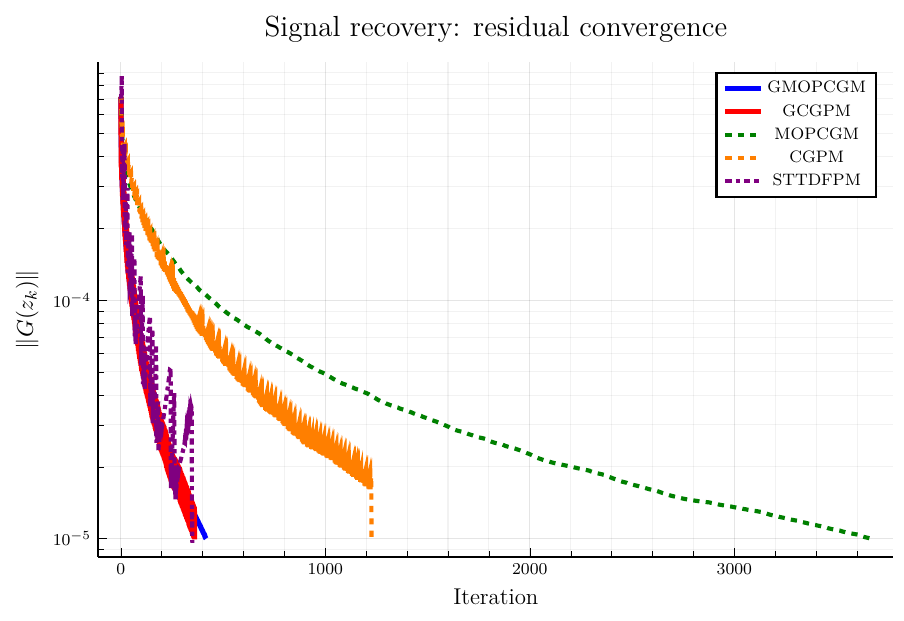}
\caption{Residual $\norm{G(z_k)}$ versus iteration for the compressed sensing instance in Figure~\ref{fig:cs_signals}. \STTDFPM{} and \GCGPM{} converge fastest; \MOPCGM{} requires significantly more iterations.}\label{fig:cs_residual}
\end{figure}

\section{Application to regularized logistic regression}\label{sec:logreg}

Unlike the compressed sensing application, where the monotone equation is linear, regularized logistic regression gives rise to a genuinely nonlinear monotone equation, providing a complementary test of the methods on a real-world machine learning task.

Binary classification---assigning labels to data based on observed features---is a fundamental problem in machine learning~\cite{chang2011libsvm}. Given $N$ labeled data pairs $(a_i, b_i) \in \R^n \times \{-1,+1\}$, $i = 1,\ldots,N$, the regularized logistic regression problem seeks a weight vector~$x$ that minimizes
\begin{equation}\label{eq:logreg}
  \min_{x \in \Gamma}\; f(x) = \frac{1}{N}\sum_{i=1}^{N}
  \log\!\bigl(1 + \exp(-b_i\,a_i^T x)\bigr) + \frac{\mu}{2}\norm{x}^2,
\end{equation}
where $\mu > 0$ is an $\ell_2$-regularization parameter promoting generalization, and $\Gamma = \{x \in \R^n : \norm{x}_\infty \leq C\}$ is a box constraint that bounds the model weights. Constraining the weights to a bounded set prevents numerical overflow in the exponential, serves as an implicit $\ell_\infty$-norm penalty that complements the $\ell_2$ term, and is standard practice in large-scale machine learning for improving generalization and numerical stability. The gradient of~$f$ is
\begin{equation}\label{eq:logreg_grad}
  G(x) \;:=\; \nabla f(x) \;=\; -\frac{1}{N}\sum_{i=1}^{N}
  \frac{b_i\,a_i}{1 + \exp(b_i\,a_i^T x)} \;+\; \mu\, x.
\end{equation}
The Hessian $\nabla^2 f(x) = \frac{1}{N}\sum_{i=1}^N \sigma_i(1 - \sigma_i)\,a_i a_i^T + \mu I$, where $\sigma_i = (1 + \exp(-b_i a_i^T x))^{-1}$, satisfies $\nabla^2 f(x) \succeq \mu I \succ 0$ for all~$x$. Therefore $f$ is strongly convex, $G$ is strictly monotone, and $G$ is Lipschitz continuous on the compact set~$\Gamma$ with constant $L \leq \frac{1}{4N}\norm{A^TA} + \mu$ (using $\sigma(1-\sigma) \leq 1/4$). Consequently, Assumptions~\ref{A1}--\ref{A3} are satisfied and the problem fits the framework of~\eqref{eq:problem}. Since all three assumptions hold, the convergence guarantees of both Theorem~\ref{thm:convergenceGMOP} and Theorem~\ref{thm:convergenceGCGPM} apply. We note that \GMOPCGM{} requires only Assumptions~\ref{A1}--\ref{A2} for convergence and would therefore remain applicable even for non-Lipschitz variants of this problem, such as replacing the smooth $\ell_2$ penalty with a non-smooth regularizer. The projection onto~$\Gamma$ reduces to componentwise clamping: $\PG(x)_j = \max\{-C,\, \min\{x_j,\, C\}\}$.

Following the experimental setup in~\cite{ibrahim2025relaxed}, we evaluate the five methods on binary classification datasets from the LIBSVM repository~\cite{chang2011libsvm}. Table~\ref{tab:libsvm_datasets} lists the datasets used, which range from $n = 122$ to $n = 2{,}000$ features and from $N = 62$ to $N = 47{,}272$ samples. All features are standardized to zero mean and unit variance prior to optimization. We set $\mu = 0.1$, $C = 10$, and generate initial points as $x_0 = 4(\xi - 0.5\cdot\mathbf{1}_n)$ where $\xi \sim \mathrm{Uniform}([0,1]^n)$. Each method uses the same parameters as in the benchmark experiments (Table~\ref{tab:params}), the same tolerance $\varepsilon = 10^{-11}$, and a maximum of 5{,}000 iterations. For each dataset, we run 5 independent trials with different random seeds and report median values.

\begin{table}[htbp]
\centering
\caption{LIBSVM datasets used in the logistic regression experiments.}\label{tab:libsvm_datasets}
\begin{tabular}{clrr}
\toprule
No.\ & Dataset & $N$ (samples) & $n$ (features) \\
\midrule
1  & a1a.t        & 30,956 & 123   \\
2  & a2a.t        & 30,296 & 123   \\
3  & a3a.t        & 29,376 & 123   \\
4  & a4a.t        & 27,780 & 123   \\
5  & a5a.t        & 26,147 & 123   \\
6  & a6a.t        & 21,341 & 123   \\
7  & a7a.t        & 16,461 & 123   \\
8  & a8a.t        & 9,865  & 122   \\
9  & a9a.t        & 16,281 & 122   \\
10 & colon-cancer & 62     & 2,000 \\
11 & w1a.t        & 47,272 & 300   \\
12 & w2a.t        & 46,279 & 300   \\
\bottomrule
\end{tabular}
\end{table}

Table~\ref{tab:logreg} summarizes the median performance across all 12 datasets and 5 trials (60 instances per method). Both \GMOPCGM{} and \GCGPM{} converge on all 60 instances, as does \STTDFPM. \MOPCGM{} fails on one colon-cancer instance (98.3\%), while \CGPM{} fails on four colon-cancer instances (93.3\%)---the only dataset where $n = 2{,}000$ features far exceed the $N = 62$ samples. In terms of efficiency, \STTDFPM{} achieves the fewest median iterations (150) and function evaluations (426), followed closely by \GCGPM{} (184 iterations). \GMOPCGM{} reduces the median iteration count from 2{,}447 (\MOPCGM) to~313, confirming the benefit of the spectral generalization on this nonlinear problem. All converging methods reach the same classification accuracy (approximately 73.6\% on the adult-income datasets), confirming that they converge to the same optimal solution.

\begin{table}[htbp]
\centering
\sisetup{exponent-mode=scientific, round-mode=places, round-precision=2}
\caption{Logistic regression: median performance over all 12 datasets and 5 trials (60 instances per method).}\label{tab:logreg}
\begin{tabular}{lrrrrrr}
\toprule
Method & Total & Conv (\%) & med IT & med FE & med CPU (s) & med $\norm{G(x^*)}$ \\
\midrule
\GMOPCGM & 60 & 100.0 & 313 & 940 & 1.17 & \num{9.90e-12} \\
\GCGPM   & 60 & 100.0 & 184 & 676 & 0.78 & \num{9.26e-12} \\
\MOPCGM  & 60 & 98.3  & 2447 & 4894 & 9.40 & \num{9.97e-12} \\
\CGPM    & 60 & 93.3  & 626 & 1407 & 2.66 & \num{9.83e-12} \\
\STTDFPM & 60 & 100.0 & 150 & 426 & 0.75 & \num{9.40e-12} \\
\bottomrule
\end{tabular}
\end{table}

\section{Conclusion}\label{sec:conclusion}

We have proposed two derivative-free projection methods, \GMOPCGM{} and \GCGPM, for solving large-scale nonlinear monotone equations subject to convex constraints. Both methods generalize existing conjugate gradient frameworks---the modified optimal Perry method (\MOPCGM) and the conjugate gradient projection method (\CGPM)---by incorporating an adaptive spectral scaling parameter that adjusts to the local geometry of the problem. We have shown that both methods satisfy a sufficient descent condition independent of the line search, and we have established their global convergence: for \GMOPCGM{} under monotonicity alone, and for \GCGPM{} under the additional Lipschitz continuity assumption.

Numerical experiments on 18 test problems across six dimensions up to $n = 120{,}000$ demonstrate that the spectral generalization consistently improves performance over the base methods. \GCGPM{} achieves the highest convergence rate and the fewest median iterations among all five methods. Applications to compressed sensing---with a systematic sweep over sparsity, measurement ratio, and noise level---and to regularized logistic regression on real-world classification datasets further confirm the practical effectiveness of both proposed methods on problems with fundamentally different structure: the linear monotone equation arising from $\ell_1$-regularization and the nonlinear monotone equation arising from the logistic loss.

Future research directions include extending the convergence analysis of \GCGPM{} to the pseudomonotone setting (removing the Lipschitz assumption), incorporating inertial acceleration techniques~\cite{jian2022inertial}, and investigating adaptive parameter selection strategies based on online learning.

\section*{Declarations}

\textbf{Conflict of interest:} The authors declare that they have no conflict of interest.

\section*{Data Availability}
The benchmark and compressed sensing data were generated algorithmically and are fully reproducible from the methods described in the manuscript. The logistic regression experiments use publicly available binary classification datasets from the LIBSVM repository~\cite{chang2011libsvm}, accessible at \url{https://www.csie.ntu.edu.tw/~cjlin/libsvmtools/datasets/}. No new datasets were created. The code used to generate the results can be made available upon reasonable request.

\section*{Funding}
This research did not receive any specific grant from funding agencies in the public, commercial, or not-for-profit sectors.

\section*{AI Use Declaration}
During the preparation of this work the authors used Claude (Anthropic) in order to assist with manuscript editing, including tightening prose, verifying LaTeX formatting, and checking internal consistency of cross-references and notation. After using this tool, the authors reviewed and edited the content as needed and take full responsibility for the content of the publication.

\section*{Acknowledgment}
The authors gratefully acknowledge the institutional support provided by King Fahd University of Petroleum \& Minerals (KFUPM). The authors also acknowledge the support of the Interdisciplinary Research Center for Smart Mobility and Logistics (IRC-SML), KFUPM, which facilitated the research environment in which this work was conducted.

\printbibliography

\end{document}